\documentclass[12 pt]{amsart}

\usepackage{graphicx}
\usepackage{stmaryrd}
\usepackage{amssymb}

\newcommand{\ga}{\alpha}

\newcommand{\gd}{\delta}
\newcommand{\gw}{\omega}

\newcommand{\gS}{\Sigma}
\newcommand{\gs}{\sigma}
\newcommand{\eps}{\varepsilon}

\newcommand{\Q}{\mathbb{Q}}

\newcommand{\proj}{\mathrm{proj}}

\newcommand{\ssplit}{\mathtt{SPL}}

\newcommand{\har}{\upharpoonright}
\newcommand{\E}{\mathcal{E}}

\newcommand{\cantor}{2^\gw}
\newcommand{\baire}{\gw^\gw}
\newcommand{\bintree}{2^{<\gw}}
\newcommand{\gwtree}{\gw^{<\gw}}

\newcommand{\ana}{\mathbf{\Sigma}^1_1}
\newcommand{\coana}{\mathbf{\Pi}^1_1}

\newcommand{\dotxgen}{\dot g}

\newcommand{\Coll}{\mathrm{Coll}}

\newcommand{\dom}{\mathrm{dom}}

\newcommand{\rng}{\mathrm{rng}}
\newcommand{\power}{\mathcal{P}}
\newcommand{\pioneoneonsigmaoneone}{${\mathbf{\Pi}}^1_1$ on ${\mathbf{\gS}}^1_1$}

\newcommand{\trace}{\mathrm{cl}}

\newcommand{\rank}{\mathrm{rk}}
\newcommand{\nwd}{\mathtt{NWD}}

\newcommand{\basis}{\mathcal{O}}
\newcommand{\ord}{\mathrm{Ord}}

\newcommand{\cl}{\mathrm{cl}}
\newcommand{\spl}{\mathrm{split}}
\newcommand{\suc}{\mathrm{succ}}

\newcommand{\gdelta}{\mathbf{G}_\delta}

\newcommand{\fsigmadelta}{\mathbf{F}_{\sigma\delta}}

\newcommand{\btree}{\omega^{<\omega}}
\newcommand{\ro}{\mathrm{ro}}
\newcommand{\st}{\mathrm{st}}

\newtheorem{theorem}{Theorem}[section]
\newtheorem{lemma}[theorem]{Lemma}
\newtheorem*{corollary}{Corollary}
\newtheorem{proposition}[theorem]{Proposition}
\newtheorem*{claim}{Claim}

\theoremstyle{definition}
\newtheorem{definition}[theorem]{Definition}

\newtheorem{question}[theorem]{Question}
\newtheorem{conjecture}[theorem]{Conjecture}

\author{Marcin Sabok}\thanks{The first author was partially
  supported by the Mittag-Leffler Institute (Djursholm,
  Sweden) and by the ESF program ``New Frontiers of
  Infinity: Mathematical, Philosophical and Computational
  Prospects''.}

\author{Jind\v rich Zapletal} \thanks{The second author was
  partially supported by NSF grant DMS 0300201 and
  Institutional Research Plan No. AV0Z10190503 and grant
  IAA100190902 of GA AV \v{C}R. The visit of the second
  author at Wroc\l aw University was funded by a short visit
  grant of the INFTY project of ESF}

\address{Instytut Matematyczny Uniwersytetu Wroc\l awskiego,
  pl.  Grunwaldzki $2\slash 4$, $50$-$384$ Wroc\l aw, Poland
  and Institut Mittag-Leffler, Aurav\"agen 17 SE-182 60
  Djursholm Sweden}

\email{sabok@math.uni.wroc.pl}

\address{Institute of Mathematics of the Academy of Sciences
  of the Czech Republic, \v{Z}itn\'a 25, CZ - 115 67 Praha
  1, Czech Republic and Department of Mathematics University
  of Florida, 358 Little Hall PO Box 118105 Gainesville, FL
  32611-8105, USA}

\email{zapletal@math.cas.cz}

\title{Forcing properties of ideals of closed sets}

\begin{document}



\maketitle

\begin{abstract}

  With every $\sigma$-ideal $I$ on a Polish space we
  associate the $\sigma$-ideal $I^*$ generated by the closed
  sets in $I$. We study the forcing notions of Borel sets
  modulo the respective $\sigma$-ideals $I$ and $I^*$ and
  find connections between their forcing properties. To this
  end, we associate to a $\sigma$-ideal on a Polish space an
  ideal on a countable set and show how forcing properties
  of the forcing depend on combinatorial properties of the
  ideal.  For $\sigma$-ideals generated by closed sets we
  also study the degrees of reals added in the forcing
  extensions.  Among corollaries of our results, we get
  necessary and sufficient conditions for a $\sigma$-ideal
  $I$ generated by closed sets, under which every Borel
  function can be restricted to an $I$-positive Borel set on
  which it is either 1-1 or constant.
\end{abstract}

\section{Introduction}

This paper is concerned with the study of $\sigma$-ideals
$I$ on Polish spaces and associated forcing notions $P_I$ of
$I$-positive Borel sets, ordered by inclusion. If $I$ is a
$\sigma$-ideal on $X$, then by $I^*$ we denote the
$\sigma$-ideal generated by the closed subsets of $X$ which
belong to $I$. Clearly, $I^*\subseteq I$ and $I^*=I$ if $I$
is generated by closed sets.

There are natural examples when the forcing $P_I$ is well
understood, whereas little is known about $P_{I^*}$. For
instance, if $I$ is the $\sigma$-ideal of Lebesgue null
sets, then the forcing $P_I$ is the random forcing and $I^*$
is the $\sigma$-ideal $\E$. The latter has been studied by
Bartoszy\'nski and Shelah \cite{bartoszynski:closed},
\cite{bartoszynski:set} but from a slightly different point
of view. On the other hand, most classical forcing notions,
like Cohen, Sacks or Miller forcings fall under the category
of $P_I$ for $I$ generated by closed sets.

Some general observations are right on the surface. By the
results of \cite[Section 4.1]{z:book2} we have that the
forcing $P_{I^*}$ is proper and \textit{preserves Baire
  category} (for a definition see \cite[Section
3.5]{z:book2}). In the case when $I\neq I^*$ on Borel sets,
the forcing $P_{I^*}$ is not $\baire$-bounding by
\cite[Theorem 3.3.1]{z:book2}, since any condition $B\in
P_{I^*}$ with $B\in I$ has no closed $I^*$-positive subset.
It is worth noting here that the forcing $P_{I^*}$ depends
not only on the $\sigma$-ideal $I$ but also on the topology
of the space $X$.

One of the motivations behind studying the idealized forcing
notions $P_I$ is the correspodence between Borel functions
and reals added in generic extensions. The well-known
property of the Sacks or Miller forcing is that all reals in
the extension are either ground model reals, or have the
same degree as the generic real. Similar arguments also show
that the generic extensions are minimal, in the sense that
there are no intermediate models. On the other hand, the
Cohen forcing adds continuum many degrees and the structure
of the generic extension is very far from minimality. In
\cite[Theorem 4.1.7]{z:book2} the second author showed that
under some large cardinal assumptions the Cohen extension is
the only intermediate model which can appear in the $P_I$
generic extension when $I$ is universally Baire
$\sigma$-ideal generated by closed sets.

The commonly used notion of degree of reals in the generic
extensions is quite vague, however, and in this paper we
distinguish two instances.

\begin{definition}
  Let $V\subseteq W$ be a generic extension. We say that two
  reals $x,y\in W$ are of the same \textit{continuous
    degree} if there is a partial homeomorphism from
  $\baire$ to $\baire$ such that $f\in V$, $\dom(f)$ and
  $\rng(f)$ are $\gdelta$ subsets of the reals and $f(x)=y$.
  We say that $x,y\in W$ are of the same \textit{Borel
    degree} if there is a Borel automorphism $h$ of $\baire$
  such that $h\in V$ and $h(x)=y$.
\end{definition}

Following the common fashion, we say that a forcing notion
$P_I$ \textit{adds one continuous} (or \textit{Borel})
\textit{degree} if for any $P_I$ generic extension
$V\subseteq W$ any real in $W$ either belongs to $V$, or has
the same continuous (or Borel) degree as the generic real.

The following results connect the forcing properties of
$P_I$ and $P_{I^*}$. In some cases we need to make some
definability assumption, namely that $I$ is \textit{$\coana$
  on $\ana$}. For a definition of this notion see
\cite[Section 29.E]{kechris:classical} or \cite[Section
3.8]{z:book2}. Note that if $I$ is $\coana$ on $\ana$, then
$I^*$ is $\coana$ on $\ana$ too, by \cite[Theorem
35.38]{kechris:classical}.

\begin{theorem}\label{onedegree}
  If the forcing $P_I$ is proper and $\baire$-bounding, then
  the forcing $P_{I^*}$ adds one continuous degree.
\end{theorem}

\begin{theorem}\label{nocohentheorem}
  If the forcing $P_I$ is proper and does not add Cohen
  reals, then the forcing $P_{I^*}$ does not add Cohen
  reals.
\end{theorem}

\begin{theorem}\label{noindependenttheorem}
  If $I$ is $\coana$ on $\ana$ and the forcing $P_I$ is
  proper and does not add independent reals, then the
  forcing $P_{I^*}$ does not add independent reals.
\end{theorem}

\begin{theorem}\label{measuretheorem}
  If $I$ is $\coana$ on $\ana$ and the forcing $P_I$ is
  proper and preserves outer Lebesgue measure, then the
  forcing $P_{I^*}$ preserves outer Lebesgue measure.
\end{theorem}

\noindent The methods of this paper can be extended without
much effort to other cases, for example to show that if
$P_I$ is proper and has the weak Laver property, then
$P_{I^*}$ inherits this property. As a consequence, by the
results of \cite[Theorem 1.4]{z:ppoints} it follows (under
some large cardinal assumptions) that if $P_I$ proper and
preserves P-points, then $P_{I^*}$ preserves P-points as
well.

\medskip

To prove the above results we introduce a combinatorial tree
forcing notion $Q(J)$ for $J$ which is a hereditary family
of subsets of $\gw$. These are relatives of the Miller
forcing. To determine forcing properties of $Q(J)$ we study
the position of $J$ in the Kat\v etov ordering, a
generalization of the Rudin--Keisler order on ultrafilters.
Further, we show that the forcing $P_I$ gives rise to a
natural ideal $J_I$ on a countable set and we correlate
forcing properties of $Q(J_I)$ with the Kat\v etov
properies of $J_I$.  Finally, we prove that the forcing
$P_{I^*}$ is, in the nontrivial case, equivalent to
$Q(J_I)$. The conjunction of these results proves all the
above theorems.

\medskip

It is not difficult to see that the $\sigma$-ideal of meager
sets has the following maximality property: if $I$ is such
that $I^*$ is the $\sigma$-ideal of meager sets, then
$I=I^*$ on Borel sets. 

In fact, even if $P_{I^*}$ is equivalent to the Cohen
forcing, then $I=I^*$ on Borel sets.  Indeed, if the
$P_{I^*}$ generic real is a Cohen real, then $I^*$ contains
all meager sets. If $U$ is is the union of all basic open
sets in $I$, then $U\in I\cap I^*$ and if $F$ is the
complement of $U$, then on the family of Borel subsets of
$F$ the $\sigma$-ideals $I$ and $I^*$ are equal to the
$\sigma$-ideal of meager subsets of $F$.

We will show that the same holds for the $\sigma$-ideals for
the Sacks and Miller forcings.

\begin{proposition}\label{millertheorem}
  If $I$ is a $\sigma$-ideal such that $I\not=I^*$ on Borel
  sets, then $P_{I^*}$ is neither equivalent to the Miller
  nor to the Sacks forcing.
\end{proposition}

Next, motivated by the examples of the Sacks and the Miller
forcing we prove the following.

\begin{theorem}\label{degrees}
  Let $I$ be a $\sigma$-ideal generated by closed sets on a
  Polish space $X$. Any real in a $P_I$-generic extension is
  either a ground model real, a Cohen real, or else has the
  same Borel degree as the generic real.
\end{theorem}

\begin{corollary}
  Let $I$ be a $\sigma$-ideal generated by closed sets on a
  Polish space $X$. The following are equivalent:
  \begin{itemize}
  \item $P_I$ does not add Cohen reals,
  \item for any $B\in P_I$ and any continuous function
    $f:B\rightarrow\baire$ there is $C\subseteq B$, $C\in
    P_I$ such that $f$ is $1$-$1$ or constant on $C$.
  \end{itemize}
\end{corollary}

\medskip

This paper is organized as follows. In Section
\ref{sec:trees} we introduce the tree forcing notions $Q(J)$
and relate their forcing properties with the Kat\v etov
properties of $J$. In Section \ref{sec:closure} we show how
to assciate an ideal $J_I$ to a $\sigma$-ideal $I$ and how
forcing properties of $P_I$ determine Kat\v etov properties
of $J$. In Section \ref{sec:representation} we show that in
the nontrivial case the forcing notions $P_{I^*}$ and
$Q(J_I)$ are equivalent. In Section \ref{sec:miller} we
prove Proposition \ref{millertheorem}. In Section
\ref{sec:bordeg} we prove Theorem \ref{degrees}.

\section{Notation}

The notation in this paper follows the set theoretic
standard of \cite{jech:set}. Notation concerning idealized
forcing follows \cite{z:book2}. 

For a poset $P$ we write $\ro(P)$ for the Boolean algebra of
regular open sets in $P$. For a Boolean algebra $B$ we write
$\st(B)$ for the Stone space of $B$.  If $\lambda$ is a
cardinal, then $\Coll(\omega,\lambda)$ stands for the poset
of finite partial functions from $\omega$ into $\lambda$,
ordered by inclusion.

If $T\subseteq Y^{<\omega}$ is a tree and $t\in T$ is a
node, then we write $T\restriction t$ for the tree $\{s\in
T: s\subseteq t\ \vee\ t\subseteq s\}$.  For $t\in T$ we
denote by $\suc_T(t)$ the set $\{y\in Y: t^\smallfrown y\in
T\}$.  We say that $t\in T$ is a \textit{splitnode} if
$|\suc_T(t)|>1$. The set of all splitnodes of $T$ is denoted
by $\spl(T)$.

\section{Combinatorial tree forcings}\label{sec:trees}

In this section we assume that $J$ is a family of subsets of
a countable set $\dom(J)$. We assume that $\omega\notin J$
and that $J$ is \textit{hereditary}, i.e.  if $a\subseteq
b\subseteq\dom(J)$ and $b\in J$, then $a\in J$.
Occasionally, we will require that $J$ is an ideal. We say
that $a\subseteq\dom(J)$ is \textit{$J$-positive} if
$a\notin J$.  For a $J$-positive set $a$ we write
$J\restriction a$ for the family of all subsets of $a$ which
belong to $J$.

\begin{definition}
  The poset $Q(J)$ consists of those trees
  $T\subseteq\dom(J)^{<\omega}$ for which every node $t\in
  T$ has an extension $s\in T$ satisfying $\suc_T(s)\not\in
  J$. $Q(J)$ is ordered by inclusion.
\end{definition}

Thus the Miller forcing is just $Q(J)$ when $J$ is the
Fr\'echet ideal on $\gw$. $Q(J)$ is a forcing notion adding
the generic branch in $\dom(J)^\omega$, which also
determines the generic filter. We write $\dot g$ for the
canonical name for the generic branch. Basic fusion
arguments literally transfer from the Miller forcing case to
show that $Q(J)$ is proper and preserves the Baire category.

\begin{proposition}
  The forcing $Q(J)$ is equivalent to a forcing $P_I$ where
  $I$ is a $\sigma$-ideal generated by closed sets.
\end{proposition}

\begin{proof}
  To simplify notation assume $\dom(J)=\omega$. Whenever
  $f:\gwtree\to J$ is a function, let
  $A_f=\{x\in\baire:\forall n<\omega\ \ x(n)\in
  f(x\restriction n)\}$.  Note that the sets $A_f$ are
  closed. Let $I_J$ be the $\gs$-ideal generated by all sets
  of this form.

  \begin{lemma}
    An analytic set $A\subseteq\baire$ is $I_J$-positive if
    and only if it contains all branches of a tree in
    $Q(J)$.
  \end{lemma}

  \begin{proof}
    For a set $C\subseteq\baire\times\baire$ we consider the
    game $G(C)$ between Players I and II in which at $n$-th
    round Player I plays a finite sequence $s_n\in\btree$
    and a number $m_n\in\omega$, and Player II answers with
    a set $a_n\in J$. The first element of the sequence
    $s_{n+1}$ must not belong to the set $a_n$. In the end
    let $x$ be the concatenation of $s_n$'s and let $y$ be
    the concatenation of $m_n$'s. Player I wins if $\langle
    x,y\rangle\in C$.
    \begin{claim}
      Player II has a winning strategy in $G(C)$ if and only
      if $\proj(C)\in I_J$. If Player I has a winning
      strategy in $G(C)$, then $\proj(C)$ contains all
      branches of a tree in $Q(J)$.
    \end{claim}
    The proof of the above Claim is standard (cf.
    \cite[Theorem 21.2]{kechris:classical}) and we omit it.
    Now, if $C\subseteq\baire\times\baire$ is closed such
    that $\proj(C)=A$, then determinacy of $G(C)$ gives the
    desired property of $A$.
  \end{proof}
  This shows that $P_{I_J}$ has a dense subset isomorphic to
  $Q(J)$, so the two forcing notions are equivalent.
\end{proof}

If $J$ is coanalytic, then the $\gs$-ideal $I_J$ associated
with the poset $Q(J)$ is $\coana$ on $\ana$. The further,
finer forcing properties of $Q(J)$ depend on the position of
$J$ in the Kat\v etov ordering.

\begin{definition}[\cite{katetov:filters}]
  Let $H$ and $F$ be hereditary families of subsets of
  $\dom(H)$ and $\dom(F)$ respectively. $H$ is \emph{Kat\v
    etov above} $F$, or $H\geq_K F$, if there is a function
  $f:\dom(H)\rightarrow\dom(F)$ such that $f^{-1}(a)\in H$
  for each $a\in F$.
\end{definition}

\noindent For a more detailed study of this order see
\cite{hrusak:survey}. It turns out that for many
preservation-type forcing properties $\phi$ there is a
critical hereditary family $H_\phi$ such that $\phi(Q(J))$
holds if and only if $J\restriction a\not\geq_K H_\phi$ for
every $a\notin J$.  This section collects several results of
this kind.

\begin{definition}
  We say that $a\subseteq \bintree$ is \textit{nowhere dense}
  if every finite binary sequence has an extension such that
  no further extension falls into $a$. $\nwd$ stands for the
  ideal of all nowhere dense subsets of $\bintree$.
\end{definition}

\begin{theorem}\label{nocohen}
  $Q(J)$ does not add Cohen reals if and only if
  $J\restriction a\not\geq_K\nwd$ for every $J$-positive set
  $a$.
\end{theorem}

\begin{proof}
  On one hand, suppose that there exists a $J$-positive set
  $a$ such that $J\restriction a\geq_K \nwd$ as witnessed by
  a function $f:a\to\bintree$. Then, the tree $a^{<\gw}$
  forces the concatenation of the $f$-images of numbers on
  the generic sequence to be a Cohen real.

  On the other hand, suppose that $J\restriction
  a\not\geq_K\nwd$. Let $T\in Q(J)$ be a condition and $\dot
  y$ be a name for an infinite binary sequence. We must show
  that $\dot y$ is not a name for a Cohen real. That is, we
  must produce a condition $S\leq T$ and an open dense set
  $O\subseteq\cantor$ such that $S\Vdash\dot y\notin\check
  O$.

  Strengthening the condition $T$ if necessary we may assume
  that there is a continuous function $f:[T]\to\cantor$ such
  that $T\Vdash \dot y=\dot f(\dotxgen)$. For every
  splitnode $t\in T$ and for every $n\in \suc_T(t)$ pick a
  branch $b_{t,n}\in [T]$ such that $t^\smallfrown
  n\subseteq b_{t, n}$. Use the Kat{\v e}tov assumption to
  find a $J$-positive subset $a_t\subseteq \suc_T(t)$ such
  that the set $\{f(b_{t,n}):n\in a_t\}\subseteq\cantor$ is
  nowhere dense.

  Consider the countable poset $P$ consisting of pairs
  $p=\langle s_p, O_p\rangle$ where $s_p$ is a finite set of
  splitnodes of $T$, $O_p\subseteq\cantor$ is a clopen set,
  and $O_p\cap \{f(b_{t,n}):t\in s_p, n\in a_t\}=\emptyset$.
  The ordering is defined by $q\leq p$ if
  \begin{itemize}
  \item $s_p\subseteq s_q$ and $O_p\subseteq O_q$,
  \item if $t\in s_q\setminus s_p$, then $f(x)\notin O_p$
    for each $x\in[T]$ such that $t\subseteq x$.
  \end{itemize}
  Choose $G\subseteq P$, a sufficiently generic filter, and
  define $O=\bigcup_{p\in G}O_p$ and $S\subseteq T$ to be
  the downward closure of $\bigcup_{p\in G}s_p$. Simple
  density arguments show that $O\subseteq\cantor$ is open
  dense and moreover, $S\in Q(J)$, since for every node
  $t\in\bigcup_{p\in G}s_p$ and every $n\in a_t$ we have
  $t^\smallfrown n\in S$. The definitions show that
  $f''[S]\cap O=\emptyset$ as desired.
\end{proof}

\begin{definition}
  Let $0<\eps<1$ be a real number. The ideal $S_\eps$ has as
  its domain all clopen subsets of $\cantor$ of Lebesgue
  measure less than $\eps$, and it is generated by those
  sets $a$ with $\bigcup a\not = \cantor$.
\end{definition}

\noindent This ideal is closely connected with the Fubini
property of ideals on countable sets, as shown below in a
theorem of Solecki.

\begin{definition}
  If $a\subseteq\dom(J)$ and $D\subseteq a\times\cantor$,
  then we write $$\int_a D\ dJ=\{y\in\cantor: \{j\in a:\
  \langle j,y\rangle\notin D\}\in J\}.$$ $J$ has the
  \emph{Fubini property} if for every real $\eps>0$, every
  $J$-positive set $a$ and every Borel set $D\subseteq
  a\times\cantor$ with vertical sections of Lebesgue measure
  less than $\eps$, the set $\int_a D\ dJ$ has outer measure
  at most $\eps$.
\end{definition}

\noindent Obviously, the ideals $S_\eps$ as well as all
families above them in the Kat\v etov ordering fail to have
the Fubini property. The following theorem implicitly
appears in \cite[Theorem 2.1]{solecki:filters}, the
formulation below is stated in \cite[Theorem
3.13]{hrusak:survey} and proved in \cite[Theorem
3.7.1]{meza:phd}.

\begin{theorem}[Solecki]
  Suppose $F$ is an ideal on a countable set. Then either
  $F$ has the Fubini property, or else for every (or
  equivalently, some) $\eps>0$ there is a $F$-positive set
  $a$ such that $F\restriction a\geq_K S_\eps$.
\end{theorem}

\noindent By $\mu$ we denote the outer Lebesgue measure on
$\cantor$. For a definition of \textit{preservation of outer
  Lebesgue measure} and further discussion on this property
see \cite[Section 3.6]{z:book2}.

\begin{theorem}\label{fubini}
  Suppose that $J$ is a universally measurable ideal. $Q(J)$
  preserves outer Lebesgue measure if and only if $J$ has
  the Fubini property.
\end{theorem}

\begin{proof}
  Suppose on one hand that $J$ fails to have the Fubini
  property. Find a sequence of $J$-positive sets $\langle
  b_n:n\in\gw\rangle$ such that $J\restriction b_n\geq_K
  S_{2^{-n}}$, as witnessed by functions $f_n$.  Consider
  the tree $T$ of all sequences $t\in\dom(J)^{<\gw}$ such
  that $t(n)\in b_n$ for each $n\in\dom(t)$. Let $\dot B$ be
  a name for the set $\{z\in\cantor:\exists^\infty n\ z\in
  f_n(\dotxgen(n))\}$.  $T$ forces that the set $\dot B$ has
  measure zero, and the definition of the ideals $S_\eps$
  shows that every ground model point in $\cantor$ is forced
  to belong to $\dot B$.  Thus $Q(J)$ fails to preserve
  Lebesgue outer measure at least below the condition $T$.

  On the other hand, suppose that the ideal $J$ does have
  the Fubini property. Suppose that $Z\subseteq\cantor$ is a
  set of outer Lebesgue measure $\gd$, $\dot O$ is a
  $Q(J)$-name for an open set of measure less or equal to
  $\eps<\gd$, and $T\in Q(J)$ is a condition.  We must find
  a point $z\in Z$ and a condition $S\leq T$ forcing $\check
  z\notin \dot O$.

  By a standard fusion argument, thinning out the tree $T$
  if necessary, we may assume that there is a function
  $h:\spl(T)\to\basis$ such that $$T\Vdash\dot O=\bigcup\{
  h(\dotxgen\restriction n+1):\ \dotxgen\restriction
  n\in\spl(T)\}.$$ Moreover, we can make sure that if
  $t_n\in T$ is the $n$-th splitting node, then
  $T\restriction {t_n}$
  decides a subset of $\dot O$ with measure greater than
  $\eps\slash 2^n$. Hence, if we write
  $f(t_n)=\eps\slash2^n$, then for every splitnode $t\in T$
  and every $n\in \suc_T(t)$ we have $\mu(h(t^\smallfrown
  n))<f(t)$.

  Now, for every splitnode $t\in T$ let $$D_t=\{\langle
  O,x\rangle: x\in\cantor\ \wedge\ O\in\suc_T(t)\ \wedge\
  x\in h(t^\smallfrown O)\}.$$ It follows from universal
  measurability of $J$ that the set $\int_{\suc_T(t)}D_t\
  dJ$ is measurable. It has mass not greater than $f(t)$, by
  the Fubini assumption.  Since
  $\sum_{t\in\spl(T)}f(t)<\gd$, we can find $$z\in
  Z\setminus\bigcup_{t\in \spl(T)}\int_{\suc_T(t)}D_t\ dJ.$$
  Let $S\subseteq T$ be the downward closure of those nodes
  $t^\smallfrown n$ such that $t\in T$ is a splitnode and
  $n\in \suc_T(t)$ is such that $z\notin h(t^\smallfrown
  n)$. $S$ belongs to $Q(J)$ by the choice of the point $z$
  and $S\Vdash \check z\notin\dot O$, as required.
\end{proof}

An \textit{independent real} is a set $x$ of natural numbers
in a generic extension such that both $x$ and the complement
of $x$ meet every infinite set of natural numbers from the
ground model.

\begin{definition}
  $\ssplit$ is the family of nonsplitting subsets of
  $\bintree$, i.e. those $a\subseteq\bintree$ for which
  there is an infinite set $c\subseteq\gw$ such that
  $t\restriction c$ is constant for every $t\in a$.
\end{definition}

Obviously, $\ssplit$ is an analytic set, but it is not clear
whether it is also coanalytic. 

\begin{question}
  Is $\ssplit$ a Borel set?
\end{question}

In the following theorem we show that in two quite general
cases $\ssplit$ is critical for the property of adding
independent reals.

Note that if $J$ is an ideal, $H$ is hereditary and $H'$ is
the ideal generated by $K$, then $J\leq_K H$ if and only if
$J\leq_K H'$.  Therefore, in case $J$ is an ideal,
$J\geq_K\ssplit$ is equivalent to $J$ being Kat\v etov above
the ideal generated by $\ssplit$. The latter is analytic, so
in particular it has the Baire property.

\begin{theorem}\label{splitting}
  Suppose that $J$ is coanalytic or $J$ is an ideal with the
  Baire property. $Q(J)$ does not add independent reals if
  and only if $J\restriction a\not\geq_K\ssplit$ for every
  $J$-positive $a$.
\end{theorem}

\begin{proof}
  Again, the left to right direction is easy. If
  $J\restriction a\geq_K\ssplit$ for some $J$-positive set
  $a$, as witnessed by a function $f$, then the condition
  $a^{<\gw}\in Q(J)$ forces that the concatenation of
  $\langle f(\dotxgen(n)):n\in\gw\rangle$ is an independent
  real.

  For the right to left direction, we will need two
  preliminary general facts. For a set $a\subseteq\gw$ by an
  \textit{interval in $a$} we mean a set of the form
  $[k,l)\cap a$.

  First, let $a\subseteq\gw$ be a $J$-positive set, and let
  Players I and II play a game $G(a)$, in which they
  alternate to post consequtive (pairwise disjoint) finite
  intervals $b_0, c_0, b_1, c_1,\dots$ in the set $a$.
  Player II wins if the union of his intervals
  $\bigcup_{n<\omega} c_n$ is $J$-positive.

\begin{lemma}
  Player II has a winning strategy in $G(a)$ for any
  $a\notin J$.
\end{lemma}

\begin{proof}
  In case $J$ is an ideal with the Baire property, this
  follows immediately from the Talagrand theorem
  \cite[Theorem 4.1.2]{bartoszynski:set}. Indeed, if
  $\{I_k:k<\omega\}$ is a partition of $a$ into finite sets
  such that each $b\in J$ covers only finitely many of them,
  then the strategy for II is as follows: at round $n$ pick
  $c_n$ covering one of the $I_k$'s.

  Now we prove the lemma in case $J$ is coanalytic.
  Consider a related game, more difficult for Player II. Fix
  a continuous function $f:\baire\to\power(a)$ such that its
  range consists exactly of all $J$-positive sets. The new
  game $G'(a)$ proceeds just as $G(a)$, except Player II is
  required to produce sequences $t_n\in\gwtree$ of length
  and all entries at most $n$, and in the end, Player II
  wins if $y=\bigcup_{n<\omega}t_n\in\baire$ and
  $f(y)\subseteq\bigcup_{n<\omega}c_n$.

  Clearly, the game $G'(a)$ is Borel and therefore
  determined.  If Player II has a winning strategy in
  $G'(a)$, then she has a winning strategy in $G(a)$ and we
  are done.  Thus, we only need to derive a contradiction
  from the assumption that Player I has a winning strategy
  in $G'(a)$.

  Well, suppose $\gs$ is such a winning strategy. We
  construct a strategy for Player I in $G(a)$ as follows.
  The first move $b_0=\sigma(\emptyset)$ does not change.
  Suppose Player I is going to make her move after the sets
  $b_0,c_0,\ldots,b_n,c_n$ have been chosen. For each
  possible choice of the sequences $t_m$ for $m<n$ consider
  a run of $G'(a)$ in which Player I plays according to
  $\sigma$ and Player II plays the pairs $(b_m',t_m)$, where
  $b_m'$ are the intervals $b_m$ adjusted downward to the
  previous move of Player I. The next move of Player I is
  now the union of all finitely many moves the strategy
  $\gs$ dictates against such runs in $G'(a)$. It is not
  difficult to see that this is a winning strategy for
  Player I in the original game $G$. However, Player I
  cannot have a winning strategy in the game $G$ since
  Player II could immediately steal it and win herself.
\end{proof}

Second, consider the collection $F$ of those subsets
$a\subseteq\gwtree$ such that there is no tree $T\in Q(J)$
whose splitnodes all fall into $a$.

\begin{lemma}\label{idealclaim}
  The collection $F$ is an ideal.
\end{lemma}

\begin{proof}
  The collection $F$ is certainly hereditary. To prove the
  closure under unions, let $a=a_0\cup a_1$ be a partition
  of the set of all splitnodes of a $Q(J)$ tree into two
  parts. We must show that one part contains all splitnodes
  of some $Q(J)$ tree. For $i\in 2$ build rank functions
  $\rank_i: a_i\to\ord\cup\{\infty\}$ by setting
  $\rank_i\geq 0$ and $\rank_i(t)\geq\ga+1$ if the set
  $\{n\in\gw:t^\smallfrown n$ has an extension $s$ in $a_i$
  such that $\rank_i(s)\geq\ga\}$ is $J$-positive. If the
  rank $\rank_i$ of any splitnode is $\infty$ then the nodes
  whose rank $\rank_i$ is $\infty$ form a set of splitnodes
  of a tree in $Q(J)$, contained in $a_i$.  Thus, it is
  enough to derive a contradiction from the assumption that
  no node has rank $\infty$.

  Observe that if $t\in a$ is a node with
  $\rank_i(t)<\infty$, then there is $n\in\gw$ such that $a$
  contains nodes extending $t^\smallfrown n$, but all of
  them either have rank less than $\rank_i(t)$ or do not
  belong to $a_i$. Thus, one can build a finite sequence of
  nodes on which the rank decreases and the last one has no
  extension in the set $a_i$. Repeating this procedure
  twice, we will arrive at a node of the set $a$ which
  belongs to neither of the sets $a_0$ or $a_1$, reaching a
  contradiction.
\end{proof}

Now suppose that $J\restriction a\not\geq_K\ssplit$ for every
$J$-positive set $a$. Let $T\in Q(J)$ be a condition and
$\dot y$ be a $Q(J)$-name for a subset of $\omega$. We must
prove that $\dot y$ is not a name for an independent real.
That is, we must find an infinite set $b\subseteq\gw$ as
well as a condition $S\leq T$ forcing $\dot y\restriction
\check b$ to be constant. The construction proceeds in
several steps.

First, construct a tree $T'\subseteq T$ and an infinite set
$b\subseteq\gw$ such that for every splitnode $t\in T'$
there is a bit $c_t\in 2$ such that for all but finitely
many $n\in b$, for all but finitely many immediate
successors $s$ of $t$ in $T'$ we have $$T'\restriction
s\Vdash\dot y(n)=c_t.$$ To do this, enumerate $\gwtree$ as
$\langle t_i:i\in\gw\rangle$, respecting the initial segment
relation, and by induction on $i\in\gw$ construct a
descending sequence of trees $T_i\subseteq T$, sets
$b_i\subseteq\gw$, and bits $c_{t_i}\in 2$ as follows:

\begin{itemize}
\item if $t_i$ is not a splitnode of $T_i$, then do nothing
  and let $T_{i+1}=T_i$, $b_{i+1}=b_i$ and $c_{t_i}=0$;
\item if $t_i$ is a splitnode of $T_i$, then for each $j\in
  \suc_{T_i}(t_i)$ find a tree $S_j\leq T_i\restriction
  t_i^\smallfrown j$ deciding $\dot y\restriction j$, and
  use the Kat\v etov assumption to find a $J$-positive set
  $a\subseteq\suc_{T_i}(t_i)$, a bit $c_{t_i}\in 2$, and an
  infinite set $b_{i+1}\subseteq b_i$ such that whenever
  $j\in a$ and $n\in b_{i+1}\cap j$ then $S_j\Vdash \dot
  y(n)=c_{t_i}$. Let $T_{i+1}=T_i$, except below $t_i$
  replace $T_i\restriction t_i$ with $\bigcup_{j\in a}S_j$.
\end{itemize}

\noindent In the end, let $T'=\bigcap_{i<\omega}T_i$ and let
$b$ be any diagonal intersection of the sets $b_i$.

The second step uses Lemma~\ref{idealclaim} to stabilize the
bit $c_t$. Find a condition $T''\subseteq T'$ such that for
every splitnode $t\in T''$, $c_t$ is the same value, say
$0$.

The last step contains a fusion argument. For every
splitnode $t\in T''$ fix a winning strategy $\gs_t$ for
Player II in the game $G(\suc_{T''}(t))$. By induction on
$i\in\gw$ build sets $S_i\subseteq T''$, functions $f_i$ on
$S_i$, and numbers $n_i\in b$ so that

\begin{itemize}
\item $S_0\subseteq S_1\subseteq\dots$, and in fact
  $S_{i+1}$ contains no initial segments of nodes in $S_i$
  that would not be included in $S_i$ already. The final
  condition will be a tree $S$ whose set of splitnodes is
  $\bigcup_{i<\omega}S_i$;
\item for every node $s\in S_i$, the value $f_i(s)$ is a
  finite run of the game $G(\suc_{T''}(s))$ according to the
  strategy $\gs_s$, in which the union of the moves of the
  second player equals $\{j\in\gw:\exists t\in S_i\
  s^\smallfrown j\subseteq t\}$. Moreover, $f_i(s)\subseteq
  f_{i+1}(s)\subseteq\dots$. This will ensure that every
  node in $\bigcup_{i<\omega}S_i$ in fact splits into
  $J$-positively many immediate successors in the tree $S$;
\item whenever $s\in S_i$ and $j\in\omega$ is the least such
  that $s\in S_j$, then $T''\restriction s \Vdash\forall
  k\in j\ \dot y(n_k)=0$. This will ensure that in the end
  we have $S\Vdash\forall i<\gw\ \ \dot y(n_i)=0$.
\end{itemize}

The induction step is easy to perform. Suppose that $S_i,
f_i$ and $n_j$ have been found for $j<i$. Let $n_i\in b$ be
a number such that for all $s\in S_i$ for all but finitely
many $n\in\suc_{T''}(s)$ we have $$T''\restriction
s^\smallfrown n\Vdash\dot y(n_i)=0.$$ For every node $s\in
S_i$, let $d_s$ be a finite set such that for all
$n\in\suc_{T''}(s)\setminus d_s$ and for all $j\leq i$
\begin{itemize}
\item $T''\restriction s^\smallfrown n\Vdash\dot y(n_j)=0$
\item and $s^\smallfrown n$ is not an initial segment of any
  node in $S_i$.
\end{itemize}
Extend the run $f_i(s)$ to $f_{i+1}(s)$ such that the new
moves by Player II contain no numbers in the set $d_s$.

Put into $S_{i+1}$ all nodes from $S_i$ as well as every $t$
which is the smallest splitnode of $T''$ above some
$s^\smallfrown j$ where $j$ is one of the new numbers in the
set answered by Player II in $f_{i+1}(s)$.
 
In the end put $S=\bigcup_{i<\omega} S_i$. It follows from
the construction that $S\Vdash\forall i<\omega\ \ \dot
y(n_i)=0$, as desired.

\end{proof}

We finish this section with an observation about continuous
degrees of reals in $Q(J)$ generic extensions.

\begin{definition}
  We say that $J$ has the \emph{discrete set property} if
  for every $J$-positive set $a$ and every function $f:a\to
  X$ into a Polish space, there is a $J$-positive set
  $b\subseteq a$ such that the set $f''b$ is discrete.
\end{definition}

\noindent Obviously, the discrete set property is equivalent
to being not Kat\v etov above the family of discrete subsets
of $\Q$. It is not difficult to show that it also equivalent
to being not above the ideal of those subsets of the ordinal
$\gw^\gw$ which do not contain a topological copy of the
ordinal $\gw^\gw$.

\begin{proposition}
  Suppose $J$ has the discrete set property. Then $Q(J)$
  adds one continuous degree.
\end{proposition}

\begin{proof}
  Let $T$ be a condition in $Q(J)$ and
  $f:[T]\rightarrow\baire$ a continuous function. It is
  enough to find a tree $S\in Q(J)$, $S\leq T$ such that on
  $[S]$ the function $f$ is either constant, or is a
  topological embedding. Suppose that $f$ is not constant on
  any such $[S]$. By an easy fusion argument we build
  $S\subseteq T$, $S\in Q(J)$ such that for any splitnode
  $s$ of $S$ there are pairwise disjoint open sets $U_i$ for
  $i\in\suc_S(s)$ such that $f''[S\restriction
  {s^\smallfrown i}]\subseteq U_i$ for each $i\in\suc_S(s)$.
  This implies that $f$ is a topological embedding on $[S]$.
\end{proof}

\section{Closure ideals}\label{sec:closure}

In this section $X$ is a Polish space with a complete
metric, $I$ a $\gs$-ideal on $X$ and $\basis$ a countable
topology basis for the space $X$.

\begin{definition}
  For a set $a\subseteq\basis$, define $$\trace(a)=\{x\in X:
  \forall\eps>0\ \exists O\in a\quad O\subseteq
  B_\eps(x)\},$$ where $B_\eps(x)$ stands for the ball
  centered at $x$ with radius $\eps$. We write
  $$J_I=\{a\subseteq\basis: \trace(a)\in I\}.$$
\end{definition}

\noindent It is immediate that the collection $J_I$ is an
ideal and that $J_I$ is \textit{dense}\footnote{some authors
  prefer the term \textit{tall}}, i.e. every infinite set in
$\basis$ contains an infinite subset in $J_I$. If the
$\gs$-ideal $I$ is \pioneoneonsigmaoneone, then $J_I$ is
coanalytic. On the other hand, if $X$ is compact and $J_I$
is analytic, then it follows from the Kechris Louveau Woodin
theorem \cite[Theorem 11]{klw:compact} that $J_I$ is
$\fsigmadelta$.

\begin{definition}
  An ideal $J$ on a countable set is \emph{weakly selective}
  if for every $J$-positive set $a$, any function on $a$ is
  either constant or 1-1 on a positive subset of $a$.
\end{definition}

\noindent Obviously, this is just a restatement of the fact
that the ideal is not Kat\v etov above the ideal on
$\gw\times\gw$ generated by vertical lines and graphs of
functions.

\begin{proposition}
  $J_I$ is weakly selective.
\end{proposition}

\begin{proof}
  Take a $J_I$-positive set $a$ and $f:a\rightarrow\omega$.
  Suppose that $f$ is not constant on any $J_I$-positive
  subset of $a$. We must find $b\subseteq a$ such that $f$
  is $1$-$1$ on $b$. Write $Y$ for $\trace(a)$ shrunk by the
  union of all basic open sets $U$ such that $\trace(a)\cap
  U\in I$. Enumerate all basic open sets which have nonempty
  intersection with $Y$ into a sequence $\langle
  U_n:n<\omega\rangle$. Inductively pick a sequence $\langle
  O_n\in a:n<\omega\rangle$ such that $O_n\subseteq U_n$ and
  $f(O_n)\not=f(O_i)$ for $i<n$. Suppose that $O_i$ are
  chosen for $i<n$. Let $Y_n=Y\cap U_n$. This is an
  $I$-positive set and hence $a_n=\{O\in a: O\subseteq
  U_n\}$ is $J_I$-positive. Note that $f$ assumes infinitely
  many values on $a_n$ since otherwise we could find
  $J_I$-positive $b\subseteq a_n$ on which $f$ is constant.
  Pick any $O_n\in a_n$ such that
  $f(O_n)\not\in\{f(O_i):i<n\}$. Now, the set
  $b=\{O_n:n<\omega\}$ is $J_I$-positive since $\trace(b)$
  contains $Y$.
\end{proof}

Not every ideal on a countable set can be represented as
$J_I$ for a $\sigma$-ideal $I$ on a Polish space. The
existence of such $I$ can be though of as an external
property, which brings some additional setting. It would be
interesting to find out what ``internal'' properties of an
ideal can witness existence of this ``external''
$\sigma$-ideal.

It follows from Hru\v s\'ak's Category Dichotomy
\cite[Theorem 5.20]{hrusak:survey} that if a Borel ideal $J$
is weakly selective, then it is Kat\v etov below the ideal
$\nwd$, and in fact, one can find an identification of
$\dom(J)$ and $2^{<\gw}$ so that the ideal embeds into
$\nwd$ via this identification. The ideal $\nwd$ is of the
form $J_I$ when $I$ is the $\sigma$-ideal of meager sets on
$\cantor$.  Motivated by the result of Hru\v s\'ak we
conjecture the following.

\begin{conjecture}
  If $J$ is a dense $\fsigmadelta$ weakly selective ideal on
  $\omega$, then there exists a Polish space with a
  countable base $\basis$ and a $\sigma$-ideal $I$ on $X$
  such that under some identification of $\omega$ and
  $\basis$ the ideal $J$ becomes $J_I$.
\end{conjecture}

We will now verify several Kat\v etov properties of the
ideal $J_I$ depending on the forcing properties of $P_I$.

\begin{proposition}
  Suppose that $P_I$ is a proper and $\baire$-bounding
  notion of forcing. Then the ideal $J_I$ has the discrete
  set property.
\end{proposition}

\begin{proof}
  Take a $J_I$-positive set $a$ and a function $f:a\to\Q$.
  Let $B=\trace(a)$. Let $\langle\dot O_n:n\in\gw\rangle$ be
  a sequence of $P_I$-names for open sets in $a$ such that
  $\dot O_n$ is forced to be wholly contained in the
  $2^{-n}$-neighborhood of the $P_I$-generic point in $B$.
  Passing to a subsequence and a subset of $a$ if necessary,
  we may assume that the sets $\dot O_n$ are pairwise
  distinct

  \noindent\textbf{Case 1}. Assume the values $\{f(\dot
  O_n):n\in\gw\}$ are forced not to have any point in the
  range of $f$ as a limit point. Use the $\baire$-bounding
  property of the forcing $P_I$ to find a condition
  $B'\subseteq B$, a sequence of finite sets $\langle
  a_n:n\in\gw\rangle$ and numbers $\eps_n>0$ such that

  \begin{itemize}
  \item $B'\Vdash\forall m<\omega\,\exists n<\omega\ \dot
    O_m\in\check a_n$;
  \item the collection $\{B_{\eps_n}(f(O)):O\in a_n,
    n\in\gw\}$ consists of pairwise disjoint open balls.
  \end{itemize}

  To see how this is possible, note that $B$ forces that for
  every point $y\in f''a$ there is an $\eps>0$ such that all
  but finitely many points of the sequence $\langle
  f(O_m):m\in\gw\rangle$ have distance greater than $\eps$
  from $y$.

  Now let $b=\bigcup_{n<\omega}a_n$.  Let $M$ be a countable
  elementary submodel of a large enough structure and let
  $B''\subseteq B'$ be a Borel $I$-positive set consisting
  only of generic points over $M$.  It is not difficult to
  observe that $B\subseteq\trace(b)$ and therefore the set
  $b$ is as required.

  \noindent\textbf{Case 2}. If the values $\{f(\dot
  O_n):n\in\gw\}$ can be forced to have a point in the range
  of $f$ as a limit point, then, possibly shrinking the set
  $a$ we can force the sequence $\langle f(\dot
  O_n):n\in\gw\rangle$ to be convergent and not eventually
  constant, hence discrete. Similarly as in Case 1, we find
  $b\subseteq a$ such that $f''b$ is discrete.
\end{proof}

\begin{proposition}
  Suppose that $P_I$ is a proper and outer Lebesgue measure
  preserving notion of forcing. Then $J_I$ has the Fubini
  property.
\end{proposition}

\begin{proof}
  Suppose that $\eps>0$ is a real number, $a\subseteq\basis$
  is a $J_I$-positive set, and $D\subseteq a\times\cantor$
  is a Borel set with vertical sections of measure at most
  $\eps$.  Assume for contradiction that the outer measure
  of the set $\int_aD\ dJ$ is greater than $\eps$. Let $B=
  \trace(a)$.  This condition forces that there is a
  sequence $\langle\dot O_n:n\in\gw\rangle$ of sets in $a$
  such that $O_n$ is wholly contained in in the
  $2^{-n}$-neighborhood of the generic point. Let $\dot C$
  be a name for the set $\{z\in\cantor:\exists^\infty
  n<\omega\ y\notin\dot O_n\}$.  This is a Borel set of
  measure greater than or equal to $1-\eps$. Since the
  forcing $P_I$ preserves the outer Lebesgue measure, there
  must be a condition $B'\subseteq B$ and a point $z\in
  \int_aD dJ$ such that $B'\Vdash\check z\in\dot C$.
  Consider the set $b=\{O\in a:z\notin O\}$.  The set
  $\trace(b)$ must be $I$-positive, since the condition $B'$
  forces the generic point to belong to it.  This, however,
  contradicts the assumption that $z\in\int_aD\ dJ$.
\end{proof}

Finally we examine the property of adding Cohen reals.

\begin{proposition}
  Suppose that $P_I$ is proper and does not add Cohen reals.
  Then $Q(J_I)$ does not add Cohen reals.
\end{proposition}

\begin{proof} We begin with a lemma.
  
\begin{lemma}\label{epsilonlemma}
  Let $a\subseteq\basis$ be $J_I$-positive and
  $f:a\rightarrow \cantor$ be a function. There is a
  $J_I$-positive set $b\subseteq a$ and a closed nowhere
  dense set $N\subseteq\cantor$ such that for each $\eps>0$
  if $b_\eps=\{O\in b: f(O)\in B_\eps(N)\}$, then
  $\trace(b_\eps)=\trace(b)$.
\end{lemma}

\begin{proof}
  Write $C=\trace(a)$. We pick a sequence of names $\langle
  \dot O_n:n<\omega\rangle$ for elements of $a$ such that
  $C$ forces that
  \begin{itemize}
  \item $\dot O_n$ is contained in the
    $2^{-n}$-neighborhood of the generic point,
  \item the sequence $\langle f(\dot O_n):n<\omega\rangle$
    is convergent in $\cantor$.
  \end{itemize}
  Let $\dot z$ be a name for $\lim_{n\rightarrow\infty}
  f(\dot O_n)$. Since $P_I$ does not add Cohen reals, there
  is a closed nowhere dense set $N\subseteq\cantor$ and
  Borel $I$-positive $B\subseteq C$ such that $B\Vdash\dot
  z\notin \check N$. Without loss of generality assume that
  $B$ consists only of generic points over a sufficiently
  big countable elementary submodel $M\prec H_\kappa$. Let
  $b$ be the set of all $O\in a$ such that $O$ is $\dot O_n$
  evaluated in $M[g]$ for some $n<\omega$ and some $g\in B$.
  Now if $\eps>0$, then clearly $\trace(b_\eps)=B$ since all
  but finitely many $f(\dot O_n)$ are forced into
  $B_\eps(N)$.
\end{proof}

Suppose $T\in Q(J_I)$ is a condition and $\dot x$ is a name
for a real. Without loss of generality assume that $T\Vdash
\dot x=f(\dot g)$ for some continuous function
$f:[T]\rightarrow\cantor$. For each $t\in\spl(T)$ and each
$O\in\suc_T(t)$ pick a branch $b_{t,O}\in[T]$ extending
$t^\smallfrown O$. By Lemma \ref{epsilonlemma} we can assume
that for each splitnode $t\in T$ we have $\suc_T(t)\not\in
J_I$ and there is a closed nowhere dense set
$N_t\subseteq\cantor$ such that for each $\eps>0$ we have
$$\trace(\{O\in\suc_T(t): f(b_{t,O})\in
B_\eps(N_t)\})=\trace(\suc_T(t)).$$ For each $t\in\spl(T)$
fix an enumeration $\langle V_t^n:n<\gw\rangle$ of all basic
open sets which have nonempty intersection with
$\cl(\suc_T(t))$. Enumerate all nonempty basic open subsets
of $\cantor$ into a sequence $\langle U_n:n<\gw\rangle$.

By induction on $n<\gw$, we build increasing finite sets
$S_n\subseteq\spl(T)$, decreasing trees $T_n\leq T$ and
nonempty clopen sets $U_n'\subseteq U_n$ such that for each
$n<\gw$ the following hold:
\begin{itemize}
\item $S_n\subseteq T_n$,
\item for each $t\in\spl(T_n)$ we have
  $\trace(\suc_{T_n}(T))=\trace(\suc_T(t))$,
\item for each $s\in S_n$ there is $t\supseteq s$ such that
  $t\in S_{n+1}$ and $t(|s|)\subseteq V_s^n$,
\item for each $s\in S_n$ and $O\in\suc_{T_n}(s)$ we have
  $b_{s,O}\in[T_n]$,
\item for each $s\in S_n$ for each $x\in[T_n\restriction s]$
  we have $f(x)\not\in\bigcup_{k<n}U_k'$.
\end{itemize}

We set $S_0=\emptyset$ and $T_0=T$. Suppose $S_n$ and $T_n$
are constructed. For each $s\in S_n$ find $\eps_s>0$ such
that
$$W=U_{n+1}\setminus\bigcup_{s\in
  S_n}B_{\eps_s}(N_s)\not=\emptyset$$ and
$O_s\in\suc_{T_n}(s)$ such that $O_s$ is contained in
$V_s^n$. Next, find $a_s\subseteq\suc_{T_n}(s)$ such that
$\trace(a_s)=\trace(\suc_{T_n}(s))$ and $f(b_{O,s})\in
B_{\eps_s}(N_s)$ for each $O\in a_s$. For each $s\in S_n$
and $O\in a_s\setminus\{O_s\}$ find $k_{s,O}<\gw$ such that
$$f''[T_n\restriction(b_{s,O}\restriction
k_{s,O})]\subseteq B_{\eps_s}(N_s).$$ Find
$U_{n+1}'\subseteq W$ which does not contain any of the
finitely many points $\{f(b_{s,O_s}): s\in S_n\}$. For each
$s\in S_n$ find $k_{s,O_s}<\gw$ such that
$$f''[T_n\restriction(b_{s,O_s}\restriction k_{s,O_s})]\cap
U_{n+1}'=\emptyset.$$

To obtain the tree $T_{n+1}$, extend each $s\in S_{n+1}$ by
$T_n\restriction(b_{s,O}\restriction k_{s,O})$ above each
$s^\smallfrown O$ for each $O\in\{O_s\}\cup a_s$. Put into
$S_{n+1}$ all nodes from $S_n$ as well as the first
splitnodes of $T_{n+1}$ above each $s^\smallfrown O_s$ for
$s\in S_n$. This ends the construction.

Now the set $U=\bigcup_{n<\gw}U_n'$ is dense open and
$T=\bigcap_{n<\gw}T_n$ is a condition in $Q(J_I)$ with the
set of splitnodes $\bigcup_{n<\gw} S_n$. By the construction
we have that $T\Vdash\dot x\notin U$, which implies that
$\dot x$ is not a name for a Cohen real. This ends the
proof.

\end{proof}

\begin{corollary}
  If $I$ is such that $P_I$ is proper and does not add Cohen
  reals, then $J_I\restriction a\not\geq_K\nwd$ for any
  $J_I$-positive set $a$.
\end{corollary}

\section{Tree representation}\label{sec:representation}

In this section we show that under suitable assumptions the
forcing $P_{I^*}$ is equivalent to the tree forcing
$Q(J_I)$.

\begin{definition}
  Let $J$ be an ideal on $\basis$ and $T\in Q(J)$. We say
  that $T$ is \textit{Luzin} if the sets on the $n$-th level
  have diameter less than $2^{-n}$ and for each $t\in T$ the
  immediate successors of $t$ in $T$ are pairwise disjoint.
  If $T$ is Luzin, then we write $\pi[T]$ for
  $\{\bigcap_{n<\omega} x(n):x\in [T]\}$.
\end{definition}

\begin{proposition}\label{positive}
  Let $I$ be a $\gs$-ideal on a Polish space $X$. If $T\in
  Q(J_I)$ is Luzin, then $\pi(T)\in P_{I^*}$.
\end{proposition}
\begin{proof}
  The set $\pi[T]$ is a 1-1 continuous image of $[T]$, which
  is a Polish space, hence $\pi[T]$ is Borel. To see that
  $\pi[T]$ is $I^*$-positive consider the function
  $\varphi:[T]\rightarrow X$ which assings to any $x\in[T]$
  the single point in $\bigcap_{n<\omega} x(n)$. Note that
  $\varphi$ is continuous since the diameters of open sets
  on $T$ vanish to $0$. Now if
  $\pi[T]\subseteq\bigcup_{n<\omega} E_n$ where each $E_n$
  is closed and belongs to $I$, then $\varphi^{-1}(E_n)$ are
  closed sets covering the space $[T]$. By the Baire
  category theorem, one of them must have nonempty interior.
  So there is $n<\omega$ and $t\in T$ such that every
  immediate successor of $t$ in $T$ belongs to
  $\varphi^{-1}(E_n)$.  Now for each $u\in \suc_T(t)$ we
  have $u\cap E_n\not=\emptyset$, which implies that
  $\trace(\suc_T(t))\subseteq E_n$ and contradicts the fact
  that $\trace(\suc_T(t))$ is $I$-positive.
\end{proof}

The following proposition, combined with the propositions
proved in the previous section, gives Theorems
\ref{onedegree}, \ref{nocohentheorem},
\ref{noindependenttheorem} and \ref{measuretheorem} from the
introduction (recall that the Cohen forcing adds an
independent real and does not preserve outer Lebesgue
measure).

\begin{proposition}\label{representation}
  Suppose $I$ is a $\gs$-ideal on a Polish space $X$ such
  that the poset $P_I$ is proper and is not equivalent to
  the Cohen forcing under any condition. For any $B\in
  P_{I^*}$
  \begin{itemize}
  \item either $I^*$ and $I$ contain the same Borel sets
    below $B$,
  \item or there is $C\in P_{I^*}$ below $B$ such that below
    $C$ the forcing $P_{I^*}$ is equivalent to $Q(J_I)$.
  \end{itemize} 
\end{proposition}

\begin{proof}
  Suppose that $B\subseteq X$ is a Borel set which belongs
  to $I$ but not to $I^*$. Assume also that $B$ forces that
  the generic point is not a Cohen real. By the Solecki
  theorem \cite[Theorem 1]{solecki:analytic}, we may assume
  that $B$ is a $\gdelta$ set and for every open set
  $O\subseteq X$, if $B\cap O\neq \emptyset$, then $B\cap
  O\notin I^*$.  Represent $B$ as a decreasing intersection
  $\bigcap_{n<\gw} O_n$ of open sets.

  We build a Luzin scheme $T$ of basic open sets $U_t$ for
  $t\in\btree$ satisfying the following demands:

  \begin{itemize}
  \item $U_t\subseteq O_{|t|}$ and $U_t\cap B\not=\emptyset$,
  \item the sets in $\suc_T(t)$ have pairwise disjoint
    closures and are disjoint from $\trace(\suc_T(t))$,
    which is an $I$-positive set.
  \end{itemize}

  To see how this is done, suppose that $U_t$ are built for
  $t\in\omega^{\leq n}$ and take any $t\in\omega^n$. The set
  $\cl(B\cap U_t)$ is $I$-positive, and since the
  $P_I$-generic real is not forced to be a Cohen real, there
  is a closed nowhere dense $I$-positive subset $C$ of
  $\cl(B\cap U_t)$. Find a discrete set $D=\{d_n:n<\omega\}$
  such that $D\subseteq B\cap U_t$ and $C\subseteq\cl(D)$.
  For each $n<\omega$ find a basic open neighborhood
  $V_n\subseteq U_t\cap O_{|t|+1}$ of $d_n$ such that the
  closures of the sets $V_n$ are pairwise disjoint, disjoint
  from $C$ and $C\subseteq\trace(\{V_n:n<\omega\})$. Put
  $U_{t^\smallfrown n}=V_n$.

  Let $T\in Q(J)$ be the Luzin scheme constructed above.
  Clearly, $T$ is Luzin, as well as each $S\in Q(J_I)$ such
  that $S\leq T$.  For each $S\leq T$ the set
  $\pi(S)\subseteq \pi(T)$ is Borel and $I^*$-positive by
  Proposition \ref{positive}.  We will complete the proof by
  showing that the range of $\pi$ is a dense subset of
  $P_{I^*}$ below the condition $\pi(T)$.

  For $C\subseteq B$ which is an $I^*$-positive set we must
  produce a tree $S\in Q(J), S\subseteq T$, such that
  $\pi[S]\subseteq C$. By the Solecki theorem we may assume
  that the set $C$ is $\gdelta$, a decreasing intersection
  $\bigcap_{n<\omega} W_n$ of open sets and for every open
  set $O\subseteq X$ if $O\cap C\neq \emptyset$, then $O\cap
  C\notin I$.  

  By tree induction build a tree $S\subseteq T$ such that
  for every sequence on $n$-th splitting level, the last set
  on the sequence is a subset of $W_n$, and still has
  nonempty intersection with the set $C$. In the end, the
  tree $S\subseteq T$ will be as required. 

  Now suppose that immediate successors of nodes on the
  $n$-th splitting level have been constructed.  Let $t$ be
  one of these successors. Find its extension $s\in T$ such
  that the last set $O$ on it is a subset of $W_{n+1}$ and
  still has nonempty intersection with $C$. Note that
  $$\cl(\pi[T])\subseteq\pi[T]\cup\bigcup_{u\in
    T}\trace(\suc_T(u)).$$ Since $\cl(C\cap O)\notin I$ and
  $\pi[T]\subseteq B\in I$, this means that there must be an
  extension $u$ of $s$ such that $\cl(C\cap
  O)\cap\trace(\suc_T(u))\notin I$. This can only happen if
  the set $b=\{V\in a_u:V\cap C\neq 0\}$ is $J$-positive,
  since $\cl(C\cap
  O)\cap\trace(\suc_T(u))\subseteq\trace(b)$. Put all nodes
  $\{u^\smallfrown V:V\in b\}$ into the tree $S$ and
  continue the construction.
\end{proof}

\section{The cases of Miller and Sacks}\label{sec:miller}

In this section we prove Proposition~\ref{millertheorem}.
This depends on a key property of the Miller and Sacks
forcings.

\begin{lemma}\label{millerlemma}
  Suppose $X$ is a Polish space, $B\subseteq X$ is a Borel
  set, $T$ is a Miller or a Sacks tree and $\dot x$ is a
  Miller or Sacks name for an element of the set $B$. Then
  there is $S\subseteq T$ and a closed set $C\subseteq X$
  such that $C\setminus B$ is countable, and $S\Vdash\dot
  x\in\check C$.
\end{lemma}

\begin{proof}
  For the Sacks forcing it is obvious and we can even
  require that $C\subseteq B$. Let us focus on the Miller
  case. 

  Strengthening the tree $T$ if necessary, we may assume
  that there is a continuous function $f: [T]\to B$ such
  that $T\Vdash\dot x=f(\dotxgen)$. The problem of course is
  that the set $f''[T]$ may not be closed, and its closure
  may contain many points which do not belong to the set
  $B$.

  For every splitnode $t\in T$ and for every $n\in
  \suc_T(t)$ pick a branch $b_{t,n}\in [T]$ such that
  $t^\smallfrown n\subseteq b_{t,n}$. Next, find an infinite
  set $a_t\subseteq \suc_T(t)$ such that the points
  $\{f(b_{t,n}):n\in a_t\}$ form a discrete set with at most
  one accumulation point $x_t$. For $n\in a_t$ find numbers
  $m_{t,n}\in\gw$ and pairwise disjoint open sets $O_{t,n}$
  such that $f''[T\restriction(b_{t,n}\restriction m_{t,n})]\subseteq
  O_{t,n}$.  Find a subtree $S\subseteq T$ such that for
  every splitnode $t\in S$, if $t^\smallfrown n\in S$, then
  $n\in a_t$ and the next splitnode of $S$ past
  $t^\smallfrown n$ extends the sequence
  $b_{t,n}\restriction m_{t,n}$.

  It is not difficult to see that $\cl(f''[S])\subseteq
  f''[S]\cup\{x_t:t\in\gwtree\}$, and therefore the tree $S$
  and the closed set $C=\cl(f''[S])$ are as needed.
\end{proof}

Of course, in the previous lemma, $B$ may be in any
sufficiently absolute pointclass, like
$\mathbf{\Sigma}^1_2$. Proposition~\ref{millertheorem} now
immediately follows.

\begin{proof}[Proof of Proposition~\ref{millertheorem}]
  If the $\gs$-ideal $I$ does not contain the same Borel
  sets as $I^*$, then any condition $B\in I\setminus I^*$
  forces in $P_{I^*}$ the generic point into $B$ but outside
  of every closed set in the $\sigma$-ideal $I$.  However,
  by Lemma \ref{millerlemma} we have that if the Miller or
  the Sacks forcing forces a point into a Borel set in a
  $\gs$-ideal, then it forces that point into a closed set
  in that $\gs$-ideal. Thus, $P_{I^*}$ cannot be in the
  forcing sense equivalent neither to Miller nor to Sacks
  forcing in the case that $I\not=I^*$.
\end{proof}

\section{Borel degrees}\label{sec:bordeg}

In this section we prove Theorem \ref{degrees}. To this
end, we need to learn how to turn Borel functions into
functions which are continuous and open.

\begin{lemma}\label{generics}
  Suppose $I$ is a $\sigma$-ideal on a Polish space $X$ such
  that $P_I$ is proper. Let $B$ be Borel $I$-positive, and
  $f:B\rightarrow\baire$ be Borel. For any countable
  elementary submodel $M\prec H_\kappa$ the set $$f''\{x\in
  B: x\mbox{ is }P_I\mbox{-generic over }M\}$$ is Borel.
\end{lemma}

\begin{proof}
  Without loss of generality assume that $B=X$. Let $\dot y$
  be a $P_I$-name for $f(\dot g)$, where $\dot g$ is the
  canonical name for the generic real for $P_I$. Take
  $R\subseteq \ro(P_I)$ the complete subalgebra generated by
  $\dot y$. Notice that for each $y\in\baire$ we have $$y\in
  f''C\quad\mbox{ iff }\quad y\mbox{ is }R\mbox{-generic
    over }M.$$ Hence, it is enough to prove that
  $C'=\{y\in\baire: y\mbox{ is }R\mbox{-generic over }M\}$
  is Borel. $C'$ is a 1-1 Borel image of the set of
  ultrafilters on $R\cap M$ which are generic over $M$.  The
  latter set is $\gdelta$, so $C'$ is Borel.
\end{proof}

Now we show that Borel functions can turned into continuous
and open functions after restriction their domain and some
extension of topology. If $Y$ is a Polish space and $I$ is a
$\sigma$-ideal on $Y$, then we say that $Y$ is
\textit{$I$-perfect} if $I$ does not contain any nonempty
open subset of $Y$.

\begin{proposition}\label{top}
  Suppose $I$ is a $\sigma$-ideal on a Polish space $X$ such
  that $P_I$ is proper. Let $B\subseteq X$ be $I$-positive,
  and $f:B\rightarrow \baire$ be Borel. There are Borel sets
  $Y\subseteq B$ and $Z\subseteq\baire$ such that $Y$ is
  $I$-positive, $f''Y=Z$ and
  \begin{itemize}
  \item $Y$ and $Z$ carry Polish zero-dimensional topologies
    which extend the original ones, preserve the Borel
    structures and the topology on $Y$ is $I$-perfect.
  \item the function $f\har Y: Y\rightarrow Z$ is continuous
    and open in the extended topologies.
  \end{itemize} 
\end{proposition}

\begin{proof}
  Fix $\kappa$ big enough and let $M\prec H_\kappa$ be a
  countable elementary submodel coding $B$ and $f$. Let
  $Y=\{x\in B: x\mbox{ is }P_I\mbox{-generic over }M\}$. By
  Lemma \ref{generics} we have that $Z=f''Y$ is Borel.

  Note that if a $\mathbf{\Sigma}^0_\alpha$ set $A$ is coded
  in $M$, then there are sets $A_n$ coded in $M$,
  $A_n\in\mathbf{\Pi}^0_{<\alpha}$ such that $A=\bigcup_n
  A_n$. Therefore, we can perform the construction from
  \cite[Theorem 13.1]{kechris:classical} and construct a
  Polish zero-dimensional topology on $X$ which contains all
  Borel sets coded in $M$. Note that the Borel sets coded in
  $M$ form a basis for this topology. Moreover, $Y$ is
  homeomorphic to the set of ultrafilters in $\st(P_I\cap
  M)$ which are generic over $M$. So $Y$ is a $\gdelta$ set
  in the extended topology. Let $\tau$ be the restriction of
  this extended topology to $Y$. The fact that $\tau$ is
  $I$-perfect on $Y$ follows directly from properness of
  $P_I$.

  Let $\sigma$ be the topology on $Z$ generated by the sets
  $f''(Y\cap A)$ and their complements, for all $A\subseteq
  B$ which are Borel and coded in $M$.

  Now we prove that $f\restriction Y$ is a continuous open
  from $(Y,\tau)$ to $(Z,\sigma)$. The fact that $f$ is open
  follows right from the definitions. Now we prove that $f$
  is continuous. Fix a cardinal $\lambda$ greater than
  $2^{2^{|P_I|}}$ and a Borel set $A$ coded in $M$.

  \begin{lemma}\label{collapse}
    Given $x\in Y$ we have
    \begin{itemize}
    \item $f(x)\in f''(A\cap Y)$ if and only if
      $$M[x]\models\Coll(\omega,\lambda)\Vdash\exists x'\,
      P_I\mbox{-generic over }M\ [x'\in A\ \wedge\
      f(x)=f(x')],$$
    \item $f(x)\not\in f''(A\cap Y)$ if and only if
      $$M[x]\models\Coll(\omega,\lambda)\Vdash\forall x'\,
      P_I\mbox{-generic over }M\ [x'\in A\ \Rightarrow\
      f(x)\not=f(x')].$$
    \end{itemize}
  \end{lemma}
  \begin{proof}
    We prove only the first part. Note that in $M$ there is
    a surjection from $\lambda$ onto the family of all dense
    sets in $P_I$ as well as sujections from $\lambda$ onto
    each dense set in $P_I$.  Therefore, if $x\in Y$ and
    $g\subseteq\Coll(\omega,\lambda)$ is generic over
    $M[x]$, then in $M[x][g]$ the formula
    $$\exists x'\, P_I\mbox{-generic over }M\ [x'\in A\
    \wedge\ f(x)=f(x')]$$ is analytic with parameters $A$,
    $f$ and a real which encodes the family $\{D\cap M: D\in
    M\mbox{ is dense in }P_I\}$ and therefore it is absolute
    between $M[x][g]$ and $V$. Hence $$M[x][g]\models\exists
    x'\, P_I\mbox{-generic over }M\ [x'\in A\ \wedge\
    f(x)=f(x')]$$ if and only if $f(x)\in f^{-1}(f''(A\cap
    Y))$.
  \end{proof}
  Now it follows from from Lemma \ref{collapse} and the
  forcing theorem that both sets $Y\cap f^{-1}(f''(A\cap
  Y))$ and $Y\cap f^{-1}(Z\setminus f''(A\cap Y))$ are in
  $\tau$. This proves that $f$ is continuous.

  We need to prove that $Z$ with the topology $\sigma$ is
  Polish. Note that it is a second-countable Hausdorff
  zero-dimensional space, so in particular metrizable. As a
  continuous open image of a Polish space, $Z$ is Polish by
  the Sierpi\'nski theorem \cite[Theorem
  8.19]{kechris:classical}.  

  The fact that $\sigma$ has the same Borel structure as the
  original one follows directly from Lemma \ref{generics}.

\end{proof}

Now we are ready to prove Theorem \ref{degrees}.

\begin{proof}[Proof of Theorem \ref{degrees}]
  Let $C\in P_I$ and $\dot x$ be a name for a real such
  that
  $$C\Vdash\dot x\mbox{ is not a Cohen real and }\dot
  x\not\in V.$$ Without loss of generality assume that $C=X$
  and $C\Vdash \dot x=f(\dot g)$ for some continuous
  function $f:X\rightarrow\baire$.  We shall find $B\in P_I$
  and a Borel automorphism $h$ of $\baire$ such that
  $$B\Vdash h(f(\dot g))=\dot g.$$

  Find Polish spaces $Y\subseteq X$ and $Z\subseteq\baire$
  as in Proposition \ref{top}. Without loss of generality
  assume that $Y=X$ and the extended topologies are the
  original ones (note that $I$ is still generated by closed
  sets in any extended topology).

  Now we construct $T\in Q(J_I)$ and a Borel automorphism
  $h$ of $\baire$. To this end we build two Luzin schemes
  $U_t\subseteq X$ and $C_t\subseteq\baire$ (for
  $t\in\btree$), both with the vanishing diameter property
  and such that
  \begin{itemize}
  \item $U_t$ is basic open and $C_t$ is closed,
  \item $f''U_t\subseteq C_t$
  \item for each $t\in\btree$ the set $\{U_{t^\smallfrown
      k}:k<\omega\}$ is $J_I$-positive.
  \end{itemize}

  We put $U_\emptyset=X$ and $C_\emptyset=\baire$. Suppose
  $U_t$ and $C_t$ are built for all $t\in\omega^{<n}$. Pick
  $t\in\omega^{n-1}$. Now $f''U_t$ is an open set. Let $K$
  be the perfect kernel of $f''U_t$. $K$ is nonempty since
  $\dot x$ is forced not to be in $V$. Hence $K$ is a
  perfect Polish space and $U_t\Vdash\dot x\in K$. Note that
  there is a closed nowhere dense $N\subseteq K$ such that
  $f^{-1}(N)$ is $I$-positive, since otherwise
  $$U_t\Vdash \dot x\mbox{ is a Cohen real in
  }K.$$ Pick such an $N$ and let $M=f^{-1}(N)$. $N$ is
  closed nowhere dense in $f''U_t$ too, so $M$ is closed
  nowhere dense in $U_t$ because $f$ is continuous and open.

  Enumerate all basic open sets in $U_t$ having nonempty
  intersection with $M$ into a sequence $\langle
  V_k:k<\omega\rangle$. Inductively pick clopen sets
  $W_k\subseteq U_t$ and $C_k\subseteq\baire$ such that
  \begin{itemize}
  \item $W_k\subseteq f^{-1}(C_k)\cap V_k$ is basic open,
  \item $C_k$ are pairwise disjoint,
  \item $f^{-1}(C_k)$ are disjoint from $M$.
  \end{itemize}
  Do this as follows. Suppose that $W_i$ and $C_i$ are
  chosen for $i<k$. Since $f^{-1}(C_i)$ are disjoint from
  $M$ and $V_k\cap M\not=\emptyset$, the set $V_k\setminus
  \bigcup_{i<k}f^{-1}(C_i)$ is a nonempty clopen set.  Pick
  $x_k\in V_k\setminus\bigcup_{i<k}f^{-1}(C_i)\setminus M$.
  Since $f(x_k)\not\in N\cup\bigcup_{i<k}C_i$, there is a
  clopen neighborhood $C_k$ of $f(x_k)$ which is disjoint
  from $N\cup\bigcup_{i<k}C_i$.  Let $W_k$ be a basic
  neighborhood of $x_k$ contained in $f^{-1}(C_k)\cap V_k$.
  Put $U_{t^\smallfrown k}=W_k$ and $C_{t^\smallfrown
    k}=C_k$. Since $M\subseteq\trace(\{ W_k:k<\omega\})$, we
  have that $\{U_{t^\smallfrown k}:k<\omega\}$ is
  $J_I$-positive.

  This ends the construction of $T\in Q(J_I)$. It is routine
  now to define a Borel automorphism $h$ of $\baire$ out of
  the sets $U_\tau$ and $C_\tau$ so that $T\Vdash \dot
  g=h(f(\dot g))$. This ends the proof.

\end{proof}

The above proof essentially uses the technique of topology
extension and works for the Borel degrees but not for
continuous degrees. For the Sacks and Miller forcing,
however, we know that there is only one continuous degree
added in the generic extension. Therefore the following
question naturally appears.

\begin{question}
  Let $I$ be a $\sigma$-ideal generated by closed sets such
  that $P_I$ does not add Cohen reals. Does $P_I$ add one
  continuous degree?
\end{question}

\bibliographystyle{plain}
\bibliography{odkazy,zapletal}
\end{document}